\documentclass[a4paper,11pt,reqno]{amsart}
\usepackage{amsmath,graphicx,graphics}
\newtheorem{lem}{Lemma}
\newtheorem{example}{Example}
\newtheorem{con}{Conjecture}
\newtheorem{theo}{Theorem}

\newtheorem{cor}{Corollary}

\newtheorem{prop}{Proposition}
\numberwithin{equation}{section}
\newcommand{\id}{\operatorname{Id}}
\newcommand{\rev}{\operatorname{rev}}
\newcommand{\Hilb}{\mathrm{Hilb}}
\newcommand{\fa}{\mathfrak{a}}
\newcommand{\fb}{\mathfrak{b}}
\newcommand{\fh}{\mathfrak{h}}
\newcommand{\diag}{\operatorname{diag}}

\title[]{Spectra and eigenvectors of the Segre transformation}

\author{Ilse Fischer and  Martina Kubitzke}
    \address{Ilse Fischer \\
    Fakult\"at f\"ur Mathematik\\
     Universit\"at Wien\\
      Nordbergstrasse 15\\
       A-1090 Wien, Austria}
    \email{ilse.fischer@univie.ac.at}

\address{Martina Kubitzke\\
FB 12 -- Institut f\"ur Mathematik\\
Goethe-Universit\"at\\
Robert-Mayer Stra\ss e 10\\
 D-60325 Frankfurt am Main, Germany}
\email{kubitzke@math.uni-frankfurt.de}
\thanks{The first author acknowledges the support by the Austrian Science Foundation FWF, START grant Y463}

\begin{document}

\maketitle

\begin{abstract}
Given two sequences $\fa=(a_n)_{n\geq 0}$ and $\fb=(b_n)_{n\geq 0}$ of complex numbers such that their generating series are of the form 
$\sum_{n\geq 0}a_n t^n=\frac{\fh(\fa)(t)}{(1-t)^{d_{\fa}}}$ and $\sum_{n\geq 0}b_n t^n=\frac{\fh(\fb)(t)}{(1-t)^{d_{\fb}}}$, 
where $\fh(\fa)(t)$ and $\fh(\fb)(t)$ are polynomials, we consider their 
Segre product $\fa\ast\fb=(a_nb_n)_{n\geq 0}$. We are interested in the bilinear transformations that compute the coefficient sequence of 
$\fh(\fa\ast\fb)(t)$ from those of $\fh(\fa)(t)$ and $\fh(\fb)(t)$, where $\sum_{n\geq 0}a_nb_n t^n=\frac{\fh(\fa\ast\fb)(t)}{(1-t)^{d_{\fa}+d_{\fb}-1}}$. The motivation to study this problem comes from commutative algebra as the Hilbert series of the Segre product of two standard graded algebras equals the Segre product of the two individual Hilbert series.
We provide an explicit description of these transformations and compute their spectra. In particular, we show that the transformation matrices are 
diagonalizable with integral eigenvalues. We also provide explicit formulae for the eigenvectors of the transformation matrices.
Finally, we present a conjecture concerning the real-rootedness of $\fh(\fa^{\ast r})(t)$ if $r$ is large enough, where $\fa^{\ast r}=\fa\ast\cdots\ast \fa$ is the $r$\textsuperscript{th} Segre product of the sequence $\fa$ and the coefficients of $\fh(\fa)(t)$ are assumed to be non-negative.
\end{abstract}

\section{Introduction}
Given a sequence $\fa=(a_n)_{n\geq 0}$ of complex numbers, we consider its formal power series $\fa(t)=\sum_{n\geq 0}a_n t^n$. We are interested in the case that $\fa(t)$ can be written as 
\begin{equation}\label{eq:powerSeries}
\fa(t)=\sum_{n\geq 0}a_nt^n=\frac{h_0(\fa)+ h_1(\fa) t + \cdots + h_{d_{\fa}-1}(\fa) t^{d_{\fa}-1}}{(1-t)^{d_{\fa}}},
\end{equation}
which is possible if and only if the sequence $\mathfrak{a}$ is given as a polynomial function in $n$ of degree less than $d_{\fa}$, see e.\,g., \cite[Theorem 4.1.1]{Stanley-EC}. 
Adapting the notation in \cite{KubitzkeWelker}, we call $\fh(\fa):=(h_0(\fa),h_1(\fa),\ldots,h_{d_{\fa}-1}(\fa))$ the \emph{$h$-vector} and $\fh(\fa)(t)=h_0(\fa)+\cdots + h_{d_{\fa}-1}(\fa) t^{d_{\fa}-1}$ the \emph{$h$-polynomial} 
of the rational series $\fa(t)$ respectively of the sequence $\fa$. 
We are interested in the generating function of the \emph{Segre product} $\fa\ast\fb:=(a_nb_n)_{n\geq 0}$ of two sequences $\fa=(a_n)_{n\geq 0}$ and $\fb=(b_n)_{n\geq 0}$. If $\mathfrak{a}$ and $\mathfrak{b}$ can be expressed as polynomials in $n$ of degree less than $d_{\fa}$ and $d_{\fb}$, respectively, this generating series, the so-called \emph{Segre series} of the sequences $\fa$ and $\fb$, is given as
\begin{equation*}
(\fa\ast \fb)(t):=\sum_{n\geq 0}(a_nb_n)t^n=\frac{h_0(\fa\ast \fb)+\cdots +h_{d_{\fa}+d_{\fb}-2}(\fa\ast \fb)t^{d_{\fa}+d_{\fb}-2}}{(1-t)^{d_{\fa}+d_{\fb}-1}}.
\end{equation*}
 Our aim is to study the transformation of the numerator polynomials of $\fa(t)$ and $\fb(t)$ into the numerator polynomial of $(\fa\ast \fb)(t)$. 
Though, in principal, this is similar to the investigation of the Veronese transformation for formal power series in \cite{BrentiWelker}, there are 
two crucial differences. First, we do not get a single transformation matrix, describing the transformation of the complete $h$-vector but a 
transformation for each entry individually. Second, in our situation the considered transformation will be bilinear rather than linear. 
However, we will see that the transformation matrices for the different coefficients can be described as particular square block submatrices of a larger rectangular matrix. It will turn out that the study of those submatrices can be performed in a coherent way.

Our original motivation for this problem comes from commutative algebra. 
By the Hilbert-Serre Theorem (see \cite[Chapter 10.4]{Eisenbud}), the Hilbert series $\Hilb(A,t)=\sum_{i\geq 0}\dim_k A_it^i$ of a standard graded $k$-algebra $A=\bigoplus_{i\geq 0}A_i$ is of the form \eqref{eq:powerSeries}, where the degree of the denominator polynomial equals the Krull dimension of $A$. 
Given two standard graded $k$-algebras $A=\bigoplus_{i\geq 0} A_i$ and $B=\bigoplus_{i\geq 0} B_i$, it is customary to consider the \emph{Segre product}
of those algebras, defined by
\begin{equation*}
A\ast B=\bigoplus_{i\geq 0} A_i\otimes_k B_i,
\end{equation*}
see e.\,g., \cite{FroebergHoa,Harris,SinghWalter}. Note that if $A$ and $B$ are polynomial rings in $r$ variables, i.\,e., $A=k[x_1,\ldots,x_r]$ and $B=k[y_1,\ldots,y_r]$, the Segre product $A\ast B$ can be viewed as the homogeneous coordinate ring of the image of the \emph{Segre embedding} 
\begin{align*}
\mathbb{P}^{r-1}\times \mathbb{P}^{r-1}&\rightarrow \mathbb{P}^{r^2-1}\\ 
((v_0:v_1:\ldots:v_{r-1}),(w_0:w_1:\ldots:w_{r-1}))&\mapsto (v_0w_0:v_0w_1:\ldots:v_0w_{r-1}:v_1w_0:\ldots:v_{r-1}w_{r-1}),
\end{align*}
\cite[Chapter 13]{Eisenbud}. Clearly, the Hilbert series of the Segre product of two algebras $A$ and $B$ is given as the Segre series of the sequences $(\dim_k A_n)_{n\geq 0}$ and $(\dim_k B_n)_{n\geq 0}$, i.\,e.,
\begin{equation*}
\Hilb(A\ast B,t)=\sum_{n\geq 0}a_nb_n t^n=\Hilb(A,t)\ast \Hilb(B,t).
\end{equation*}
In the context of generating functions, this product is also frequently referred to as the \emph{Hadamard product}. 

The paper is structured as follows. Section \ref{sect:Trafo} focuses on the description of the $h$-vector transformation. 
We will derive this transformation for each single coefficient and then show how it can be traced back to the study of a family of certain square matrices (Theorem \ref{thm:trafoMain}). In Section \ref{sect:Spectrum}, we determine the 
spectra of the transformation matrices. In particular, we show that all transformation matrices are diagonalizable and have integral eigenvalues (Theorem \ref{thm:eigenvalues}). In Section \ref{sect:Eigenspaces} we compute the eigenspaces of the transformation matrices and explicitly construct formulae for the eigenvectors. Section \ref{sect:conclusion} studies some properties of the Segre transformation and concludes with an open problem.

\section{The transformation matrix for the Segre product}\label{sect:Trafo}
Throughout this section, --~without stating this explicitly~-- we assume that $\fa=(a_n)_{n\geq 0}$ and $\fb=(b_n)_{n\geq 0}$ are sequences in $\mathbb{C}$ such that their generating series are of the form
\begin{equation}
\label{gfun:a}
\fa(t)=\sum_{n\geq 0}a_n t^n =\frac{h_0(\fa)+\cdots + h_{d_{\fa}-1}(\fa)t^{d_{\fa}-1}}{(1-t)^{d_{\fa}}},
\end{equation}
and
\begin{equation*}
\fb(t)=\sum_{n\geq 0}b_n t^n =\frac{h_0(\fb)+\cdots + h_{d_{\fb}-1}(\fb)t^{d_{\fb}-1}}{(1-t)^{d_{\fb}}}.
\end{equation*}
In particular, $n\mapsto a_n$ and $n\mapsto b_n$ are polynomial functions in $n$ of degree less than $d_{\fa}$ and $d_{\fb}$, respectively (see \cite[Theorem 4.1.1]{Stanley-EC}).
For technical reasons, we will use the convention that $a_n=b_n=0$ for $n< 0$ and $h_i(\fa)=h_i(\fb)=0$ for $i<0$ or $i\geq d_{\fa}$. 
Our aim is to derive transformation matrices that describe the effect of the Segre product on the $h$-polynomial. The matrices, we are interested in, 
are the following. 
 For a non-negative integer $t$, we set 
$$
M_{d_{\fa}}(d_{\fb},t):= \left( \binom{d_{\fa}+i+t-j-1}{i+t} \binom{d_{\fb}-i-t+j-1}{j} \right)_{0 \le i , j \le d_{\fa}-1}, 
$$
where we use the following extended definition of the binomial coefficient $\binom{n}{k}$ if $n \in \mathbb{C}$ and $k \in \mathbb{Z}$:
$$
\binom{n}{k} := 
\begin{cases}
\frac{n(n-1) \cdots  (n-k+1)}{k!}, & \quad \text{if} \quad k \ge 0 \\
0, & \quad \text{otherwise}.
\end{cases}
$$
Before stating the $h$-vector transformation for the Segre product of two sequences explicitly, we need to introduce one further notation. 
For sequences $\fa=(a_n)_{n \in \mathbb{Z}}$ and $d,n \in \mathbb{Z}$, $d$ positive, we define 
$\rev_{n,d}(\fa):=(a_n,a_{n-1},a_{n-2},\ldots,a_{n-d+1})^T$ for the ``(partial) reverse'' of the sequence $\fa$, where $v^T$ denotes the transposed vector of a given vector $v$. Given these definitions, we can state the main result of this section. 

\begin{theo}\label{thm:trafoMain}
Let $\fa=(a_n)_{n\geq 0}$ and $\fb=(b_n)_{n\geq 0}$ be sequences in $\mathbb{C}$ such that their generating series are of the form
\begin{equation*}
\fa(t)=\sum_{n\geq 0}a_n t^n =\frac{h_0(\fa)+\cdots + h_{d_{\fa}-1}(\fa)t^{d_{\fa}-1}}{(1-t)^{d_{\fa}}},
\end{equation*}
and
\begin{equation*}
\fb(t)=\sum_{n\geq 0}b_n t^n =\frac{h_0(\fb)+\cdots + h_{d_{\fb}-1}(\fb)t^{d_{\fb}-1}}{(1-t)^{d_{\fb}}}.
\end{equation*}
If $d_{\fa}\leq d_{\fb}$, then 
\begin{equation*}
h_n(\fa\ast\fb)=\begin{cases}
\rev_{n,d_{\fa}}(\fh(\fa))^T \cdot M_{d_{\fa}}(d_{\fa},0) \cdot \rev_{n,d_{\fa}}(\fh(\fb)), &\mbox{ if } 0 \le n \le d_{\fa}-1\\
\rev_{d_{\fa}-1,d_{\fa}}(\fh(\fa))^T \cdot M_{d_{\fa}}(d_{\fa},n-d_{\fa}+1) \cdot \rev_{n,d_{\fa}}(\fh(\fb)), &\mbox{ if }  d_{\fa} \le n \le d_{\fb}-1\\
\rev_{n-d_{\fb}+d_{\fa},d_{\fa}}(\fh(\fa))^T \cdot M_{d_{\fa}}(d_{\fb},d_{\fb}-d_{\fa}) \cdot\rev_{n,d_{\fa}}(\fh(\fb)), &\mbox{ if }  d_{\fb} \le n \le d_{\fa}+d_{\fb}-1.\\
\end{cases}
\end{equation*}
\end{theo}

For the proof of the above theorem we will need some preparation; the result will follow at the end of this section. 

We begin by deriving a transformation rule between the terms of the sequence $\fa$ and the coefficients of the $h$-polynomial. 
It involves the \emph{backward difference operator} $\Delta_n$, defined as $\Delta_n a_n := a_n - a_{n-1}$, and uses the fact that this operator is 
invertible on sequences $(a_n)_{n \in \mathbb{Z}}$ with $a_n=0$ if $n < 0$. Indeed, $\Delta^{-1}_n = \sum_{i=0}^{\infty} E_n^{-i}$, where $E_n$ 
denotes the \emph{shift operator}, defined as $E_n a_n := a_{n+1}$. In order to see this, it is crucial that the sum is in fact finite if applied to sequences that vanish on the negative integers. 

\begin{prop} 
\label{transform:ph}
Let $\fa=(a_n)_{n \ge 0}$ be a sequence of complex numbers whose generating function can be written as \eqref{gfun:a} and set $a_n=0$ if $n<0$. Then, for $n \in \mathbb{Z}$,
$$h_n(\fa) = \Delta_n^{d_{\fa}} a_n = \sum_{j=0}^n \binom{n-j-d_{\fa}-1}{n-j} a_j \quad 
 \text{ and }  \quad 
a_n= \Delta_n^{-d_{\fa}} h_n(\fa) = \sum_{j=0}^n \binom{n-j+d_{\fa}-1}{n-j} h_j(\fa).$$
\end{prop}

The following lemma is needed in the proof of the above proposition.

\begin{lem}
\label{lem:diff}
 Let $(a_n)_{n \in \mathbb{Z}}$ be a sequence in $\mathbb{C}$ with 
$a_{n}=0$ if $n<0$. Then, for all $d \in \mathbb{Z}$, 
$$
\Delta_n^d a_n = \sum_{j=0}^n \binom{n-j-d-1}{n-j} a_j.
$$
\end{lem}

\begin{proof}[Proof of Lemma~\ref{lem:diff}] We use the binomial theorem to see that 
$$
\Delta^d_n a_n  = (\id - E_n^{-1})^{d} a_n = \sum_{i=0}^{\infty} (-1)^{i} \binom{d}{i} a_{n-i} = \sum_{i=0}^{n} (-1)^i \binom{d}{i} a_{n-i} = \sum_{j=0}^n (-1)^{n+j} \binom{d}{n-j} a_j.
$$
The assertion follows from $\binom{d}{n-j}= (-1)^{n+j} \binom{n-j-d-1}{n-j}$.
\end{proof}

\begin{proof}[Proof of Proposition~\ref{transform:ph}]
\begin{multline*}
h_0(\fa) + h_1(\fa) t + \dots + h_{d_{\fa}-1}(\fa) t^{d-1} = (1-t)^{d_{\fa}} \sum_{n \in \mathbb{Z}} a_n t^n \\
= (1-t)^{d_{\fa}-1} \left( \sum_{n \in \mathbb{Z}} a_n t^n - \sum_{n \in \mathbb{Z}} a_n t^{n+1} \right)
= (1-t)^{d_{\fa}-1} \sum_{n \in \mathbb{Z}} (\Delta_n a_n) t^n = \cdots = 
\sum_{n \in \mathbb{Z}} (\Delta_n^{d_{\fa}} a_n) t^n
\end{multline*}
This implies $h_n(\fa) = \Delta_n^{d_{\fa}} a_n$. Since $\Delta_n$ is invertible on sequences $(a_n)_{n \in \mathbb{Z}}$ with 
$a_n=0$, if $n<0$, we have $a_n= \Delta_n^{-d_{\fa}} h_n(\fa)$. 
The rest follows from Lemma~\ref{lem:diff}.
\end{proof}

According to Proposition~\ref{transform:ph} the $h$-vector entries of the Segre product $\fa\ast\fb$ can be computed in the following way:
\begin{align}
h_n(\fa \ast \fb) &= \Delta_n^{d_{\fa}+d_{\fb}-1} (a_n \cdot b_n) = \Delta_n^{d_{\fa}+d_{\fb}-1} (\Delta_n^{-d_{\fa}} h_n(\fa) \cdot \Delta_n^{-d_{\fb}} h_n(\fb)) \notag\\
&= \sum_{i,j=0}^{n} h_i(\fa) h_j(\fb) \sum_{k=0}^{n} \binom{n-k-d_{\fa}-d_{\fb}}{n-k} \binom{k-i+d_{\fa}-1}{k-i} \binom{k-j+d_{\fb}-1}{k-j} \notag\\
&= \sum_{i=0}^{\min(n,d_{\fa}-1)} \sum_{j=0}^{\min(n,d_{\fb}-1)} h_i(\fa) h_j(\fb) \sum_{k=0}^{n} \binom{n-k-d_{\fa}-d_{\fb}}{n-k} \binom{k-i+d_{\fa}-1}{k-i} \binom{k-j+d_{\fb}-1}{k-j}.\label{eq:h}
\end{align}

The next lemma shows that the inner sum in the last expression can be simplified.

\begin{lem}\label{lem:id} Let $d_{\fa}, d_{\fb},i, j$ be positive integers with 
$0 \le i < d_{\fa}$ and $0 \le j < d_{\fb}$. Then we have
\begin{equation}\label{eq1}
\sum_{k=0}^{n} \binom{n-k-d_{\fa}-d_{\fb}}{n-k} \binom{k-i+d_{\fa}-1}{k-i} \binom{k-j+d_{\fb}-1}{k-j} = \binom{d_{\fa}+j-i-1}{n-i} \binom{d_{\fb}-j+i-1}{n-j}.
\end{equation}
\end{lem}

\begin{proof}
Substituting $l= n-k$ in the left-hand side of \eqref{eq1}, we obtain
\begin{multline*}
\sum_{l=0}^{n}\binom{l-d_{\fa}-d_{\fb}}{l}\binom{n-l-i+d_{\fa}-1}{n-l-i}\binom{n-l-j+d_{\fb}-1}{n-l-j}\\
=\sum_{l=0}^{n}(-1)^l\binom{d_{\fa}+d_{\fb}-1}{l}(-1)^{n-l-i}\binom{-d_{\fa}}{n-l-i}\binom{n-l-j+d_{\fb}-1}{d_{\fb}-1}.
\end{multline*}
For the last equality, we have applied the identity $\binom{n}{k}=(-1)^k\binom{k-n-1}{k}$ and the symmetry $\binom{n}{k}=\binom{n}{n-k}$ 
to the first two and the last binomial coefficient of the product inside the sum, respectively. (Note that with the extended definition of the  
binomial coefficient, the latter identity is applicable if $n$ is a non-negative integer.) 
We may extend the range of summation to all non-negative 
integers $l$ since $\binom{-d_{\fa}}{n-l-i}=0$ if $l>n$. 
Using the first identity once more for the last binomial coefficient it follows that the left-hand side of \eqref{eq1} equals
\begin{equation}\label{eq3}
(-1)^{n+i+d_{\fb}+1}\sum_{l=0}^{\infty}\binom{d_{\fa}+d_{\fb}-1}{l}\binom{-d_{\fa}}{n-l-i}\binom{-n+l+j-1}{d_{\fb}-1}.
\end{equation}
In order to show the claim we will now apply the following triple binomial identity from \cite[p.171, (5.28)]{GrahamKnuth}:
\begin{equation}\label{eqId}
\sum_{l=0}^{\infty} \binom{m-r+s}{l}\binom{t+r-s}{t-l}\binom{r+l}{m+t}=\binom{r}{m}\binom{s}{t},
\end{equation}
where $m$, $t\in \mathbb{Z}$, to the expression in \eqref{eq3}.

If we set $m=d_{\fb}-1+i-n$, $r=j-n-1$, $s=d_{\fa}+j-i-1$, $t=n-i$, we infer from \eqref{eqId} that \eqref{eq3} can  be written as
$$
(-1)^{n+i+d_{\fb}+1}\binom{j-n-1}{d_{\fb}-1+i-n}\binom{d_{\fa}+j-i-1}{n-i}= 
\binom{d_{\fa}+j-i-1}{n-i}
=\binom{d_{\fa}+j-i-1}{n-i}\binom{d_{\fb}-j+i-1}{n-j}.
$$
For the last equality, we have once again applied the two standard binomial identities from above.
\end{proof}

Substituting the identity from Lemma \ref{lem:id} in \eqref{eq:h}, we can infer the following simplified expression of the $n$\textsuperscript{th} entry of the $h$-vector of $\fa\ast\fb$:
\begin{equation}\label{eq:bounds}
h_n(\fa \ast \fb) = \sum_{i=\max(n-d_{\fa}+1,0)}^{n} \sum_{j=\max(n-d_{\fb}+1,0)}^n h_{n-i}(\fa) h_{n-j}(\fb) \binom{d_{\fa}+i-j-1}{i} \binom{d_{\fb}+j-i-1}{j}.
\end{equation}
In the following, we assume  without loss of generality that $d_{\fa} \le d_{\fb}$. 
On the one hand,  we have
$$d_{\fa} + i - j - 1 \ge d_{\fa} + n - d_{\fa} +1 - n -1 = 0 \mbox{ and hence } \binom{d_{\fa}+i-j-1}{i}=0, \text{ if } d_{\fa}-1<j.$$
On the other hand,
$$d_{\fb}+j-i-1 \ge d_{\fb} + n- d_{\fb} + 1 - n -1 = 0 \mbox{ and hence } \binom{d_{\fb}+j-i-1}{j}=0, \text{ if } d_{\fb}-1 < i.$$  
We can thus change the bounds of summation in \eqref{eq:bounds} and obtain
\begin{equation}
\label{eq:reduced}
h(\fa \ast \fb)_n = \sum_{i=\max(n-d_{\fa}+1,0)}^{\min(n,d_{\fb}-1)} \sum_{j=\max(n-d_{\fb}+1,0)}^{\min(n,d_{\fa}-1)}  h_{n-i}(\fa) h_{n-j}(\fb) \binom{d_{\fa}+i-j-1}{i} \binom{d_{\fb}+j-i-1}{j}.
\end{equation}
Since $h_{n-i}(\fa)=0$, if $i <n-d_{\fa}+1$ or $i >n$, and $h_{n-j}(\fb)=0$, if $j < n-d_{\fb}+1$ or $j > n$, \eqref{eq:reduced} further simplifies to 
\begin{equation}\label{eq:newBounds}
h(\fa \ast \fb)_n = \sum_{i=0}^{d_{\fb}-1} \sum_{j=0}^{d_{\fa}-1}  h_{n-i}(\fa) h_{n-j}(\fb) \binom{d_{\fa}+i-j-1}{i} \binom{d_{\fb}+j-i-1}{j}.
\end{equation}
Using the notation $\rev$ for the (partial) reverse of a vector (see the paragraph preceding Theorem \ref{thm:trafoMain}), we can reformulate the last equation. 
More precisely,  
\begin{equation}\label{eq:trafo}
h_n(\fa \ast \fb) = \rev_{n,d_{\fb}}(\fh(\fa))^T \cdot M'_{d_{\fa},d_{\fb}} \cdot \rev_{n,d_{\fa}}(\fh(\fb)),
\end{equation}
where 
$M'_{d_{\fa},d_{\fb}}:=\left( \binom{d_{\fa}+i-j-1}{i} \binom{d_{\fb}-i+j-1}{j} \right)_{0 \le i \le d_{\fb}-1, 0 \le j \le d_{\fa}-1}=(m_{i,j})_{0 \le i \le d_{\fb}-1, 0 \le j \le d_{\fa}-1}$.

\begin{example}\label{ex:trafo} If $d_{\fa}=3$ and $d_{\fb}=4$, we obtain the following matrix.
$$
M'_{3,4}=
\left(
\begin{array}{ccc}
 1 & 4 & 10 \\
 3 & 6 & 6 \\
 6 & 6 & 3 \\
 10 & 4 & 1
\end{array}
\right)
$$
\end{example}


Note that the matrix in the above example is invariant under rotation of $180^{\circ}$. This phenomenon is true in general and will be useful below. Using the notation preceding the above example, we obtain:

\begin{cor}\label{cor:sym}
Let $d_{\fa}$ and $d_{\fb}$ be positive integers and $M'_{d_{\fa},d_{\fb}}=(m_{i,j})_{0 \le i \le d_{\fb}-1, 0 \le j \le d_{\fa}-1}$. Then
\begin{equation*}
m_{i,j}=m_{d_{\fb}-1-i,d_{\fa}-1-j}
\end{equation*}
for $0\leq i\leq d_{\fb}-1$ and $0\leq j\leq d_{\fa}-1$.
\end{cor}

\begin{proof}
It is routine to check that the stated identity holds.
\end{proof}

Now, recall the definition of the matrices $M_{d_{\fa}}(d_{\fb},t)$ from the beginning of this section. We defined
$$
M_{d_{\fa}}(d_{\fb},t)= \left( \binom{d_{\fa}+i+t-j-1}{i+t} \binom{d_{\fb}-i-t+j-1}{j} \right)_{0 \le i , j \le d_{\fa}-1}.
$$
Note that these matrices are $d_{\fa} \times d_{\fa}$ block submatrices of $M'_{d_{\fa},d_{\fb}}$. More precisely, 
if $t \in \{0,1,\ldots,d_{\fb}-d_{\fa}\}$, then $M_{d_{\fa}}(d_{\fb},t)$ consists of $(t+1)$\textsuperscript{st} to the  $(t+d_{\fa})$\textsuperscript{th} row of $M'_{d_{\fa},d_{\fb}}$. Indeed, if we continue Example \ref{ex:trafo}, we obtain 
$$
M_{3}(4,0)=
\left(
\begin{array}{ccc}
 1 & 4 & 10 \\
 3 & 6 & 6 \\
 6 & 6 & 3
\end{array}
\right) \qquad \text{and} \qquad 
M_{3}(4,1)= \left(
\begin{array}{ccc}
  3 & 6 & 6 \\
 6 & 6 & 3 \\
 10 & 4 & 1
\end{array}
\right).
$$

We can finally provide the proof of Theorem \ref{thm:trafoMain}.

\begin{proof}[Proof of Theorem \ref{thm:trafoMain}]
Let $d_{\fa}\leq d_{\fb}$. We need to distinguish three cases.

{\it Case 1: $0 \le n \le d_{\fa}-1$.} According to  \eqref{eq:reduced} and since $h_{n-i}(\fa)=0$ and $h_{n-i}(\fb)=0$ if $i>n$, we have 
$$
h(\fa \ast \fb)_n= \sum_{i=0}^{n} \sum_{j=0}^{n} h_{n-i}(\fa) h_{n-j}(\fb) m_{i,j}
= \sum_{i=0}^{d_{\fa}-1} \sum_{j=0}^{d_{\fa}-1} h_{n-i}(\fa) h_{n-j}(\fb) m_{i,j} = 
\rev_{n,d_{\fa}}(\fh(\fa))^T \cdot M_{d_{\fa}}(d_{\fb},0) \cdot \rev_{n,d_{\fa}}(\fh(\fb)).
$$

{\it Case 2: $d_{\fa} \le n \le d_{\fb}-1$.} In this case, \eqref{eq:reduced} implies
\begin{align*}
h_n(\fa \ast \fb)& = \sum_{i=n-d_{\fa}+1}^{n} \sum_{j=0}^{d_{\fa}-1} h_{n-i}(\fa) h_{n-j}(\fb) m_{i,j}  =
\sum_{i=0}^{d_{\fa}-1} \sum_{j=0}^{d_{\fa}-1} h_{d_{\fa}-i-1}(\fa) h_{n-j}(\fb) m_{i+n-d_{\fa}+1,j}  \\
&=\rev_{d_{\fa}-1,d_{\fa}}(\fh(\fa))^T \cdot M_{d_{\fa}}(d_{\fa},n-d_{\fa}+1) \cdot \rev_{n,d_{\fa}}(\fh(\fb)).
\end{align*}

{\it Case 3: $d_{\fb} \le n \le d_{\fa}+d_{\fb}-1$.} By \eqref{eq:reduced} and since $h_{n-i}(\fa)=0$, if $d_{\fb}-d_{\fa} \le i \le n-d_{\fa}$, and 
$h_{n-j}(\fb)=0$, if $0 \le j \le n-d_{\fb}$, we have 
\begin{align*}
h(\fa \ast \fb)_n &= \sum_{i=n-d_{\fa}+1}^{d_{\fb}-1} \sum_{j=n-d_{\fb}+1}^{d_{\fa}-1} h_{n-i}(\fa) h_{n-j}(\fb) m_{i,j} = \sum_{i=d_{\fb}-d_{\fa}}^{d_{\fb}-1} \sum_{j=0}^{d_{\fa}-1} h_{n-i}(\fa) h_{n-j}(\fb) m_{i,j} \\
&=  \sum_{i=0}^{d_{\fa}-1} \sum_{j=0}^{d_{\fa}-1} h_{n-d_{\fb}+d_{\fa}-i}(\fa) h_{n-j}(\fb) m_{i+d_{\fb}-d_{\fa},j} =
\rev_{n-d_{\fb}+d_{\fa},d_{\fa}}(\fh(\fa))^T \cdot M_{d_{\fa}}(d_{\fb},d_{\fb}-d_{\fa})\cdot \rev_{n,d_{\fa}}(\fh(\fb)). 
\end{align*}
\end{proof}

\section{Eigenvalues of products of transformation matrices}\label{sect:Spectrum}

Now having the squares matrices $M_{d_{\fa}}(d_{\fb},t)$ in hand, we proceed with the study of their eigenvalues. First note that by 
the symmetry of the binomial coefficient, the entry in the $i$\textsuperscript{th} row and the $j$\textsuperscript{th} column of the matrix can be rewritten as
$$
\binom{d_{\fa}+i+t-j-1}{d_{\fa}-j-1} \binom{d_{\fb}-i-t+j-1}{j}.
$$
If we fix $i$, $j$ and $d_{\fa}$, the last expression is obviously a polynomial in $t$ and $d_{\fb}$ and thus we may think of 
$t$ and $d_{\fb}$ as indeterminates. In the sequel, we write $d$ instead of 
$d_{\fb}$. While working with the matrices $M_{d_{\fa}}(d,t)$, we performed some experiments with the computer algebra system Mathematica \cite{Mathematica} and discovered, that the eigenvalues of $M_{d_{\fa}}(d,t)$ seemed to follow a simple pattern. For instance, if $d_{\fa}=6$, we obtain the following list of eigenvalues:
\begin{align*}
 -1,&\\
 5+d,&\\ 
-\frac{1}{2}(4+d)(5+d),& \\
\frac{1}{6}(3+d)(4+d)(5+d),&\\
 -\frac{1}{24}(2+d) (3+d)(4+d)(5+d),&\\ 
\frac{1}{120}(1+d)(2+d)(3+d)(4+d)(5+d).&
\end{align*}
Subsequently, we computed the eigenvalues of arbitrary products of the matrices $M_{d_{\fa}}(d,t)$ and observed that this behavior 
continued. Indeed, the eigenvalues of $M_5(d_1,t_1) \cdot M_5(d_2,t_2)$ are:
\begin{align*}
1,&\\
(4+d_1)(4+d_2),&\\
\frac{1}{4} (3+d_1)(4+d_1)(3+d_2)(4+d_2),& \\
\frac{1}{36} (2+d_1)(3+d_1)(4+d_1) (2+d_2)(3+d_2)(4+d_2),& \\
\frac{1}{576} (1+d_1)(2+d_1)(3+d_1)(4+d_1) (1+d_2)(2+d_2)(3+d_2)(4+d_2).&
\end{align*}
The purpose of this section is to unravel this mystery and give an explanation for this behavior. To be more precise, we 
prove the following theorem. 

\begin{theo}\label{thm:eigenvalues}
Let $d_{\fa}, n$ be positive integers. The eigenvalues of the product of matrices
$$
M_{d_{\fa}}(d_1,t_1) \cdot M_{d_{\fa}}(d_2,t_2) \cdots M_{d_{\fa}}(d_n,t_n)
$$
are $\prod\limits_{j=1}^{n} \lambda_{d_{\fa}}(d_j,i)$, $i=0,1,\ldots, d_{\fa}-1$, 
where $\lambda_{d_{\fa}}(d,i):= (-1)^{d_{\fa}+i+1} \binom{d_{\fa}+d-1}{i}.$
\end{theo}

Since in our special setting  
the numbers $\lambda_{d_{\fa}}(d_{\fb},i)$ are all distinct for $i=0,1,\ldots,d_{\fa}-1$, 
Theorem \ref{thm:eigenvalues} immediately implies that, in this case, the matrices $M_{d_{\fa}}(d_{\fb},t)$ are diagonalizable, 
\begin{cor} Let $d_{\fa}, d_{\fb}$ be positive integers with $d_{\fa} \le d_{\fb}$. Then $M_{d_{\fa}}(d_{\fb},t)$ is diagonalizable.
\end{cor}

The proof of Theorem~\ref{thm:eigenvalues} is based on the following two lemmas.

\begin{lem} \label{lemTrafo}
Let $d_{\fa}$ be a positive integer, $m$ be a non-negative integer and  
define two column vectors\\ 
$v_m:=\left( \binom{d_{\fa}-j-1}{m} \right)^T_{0 \le j \le d_{\fa}-1}$ and $w_m(d,t):=\left((-1)^{d_{\fa}+1+m} \binom{i+t+m}{m} \binom{-d-m-1}{d_{\fa}-m-1} \right)^T_{0 \le i \le d_{\fa}-1}$.
Then $$M_{d_{\fa}}(d,t) \cdot v_m = w_m(d,t).$$
\end{lem}

\begin{proof} Since the identity in question can be seen as a set of polynomial identities in $t$, it suffices to deal with the case that $t$ is a non-negative integer. 
Let $0 \le i \le d_{\fa}-1$. Then the $i$\textsuperscript{th} entry of $M_{d_{\fa}}(d,t) \cdot v_m$ is given as
\begin{multline*}
\left( M_{d_{\fa}}(d,t) \cdot v_m\right)_i = \sum_{j=0}^{{d_{\fa}}-1} \binom{{d_{\fa}}+i+t-j-1}{i+t} \binom{d-i-t+j-1}{j}  \binom{{d_{\fa}}-j-1}{m}  \\
= \sum_{j=0}^{{d_{\fa}}-1} \binom{d-i-t+j-1}{j} \frac{({d_{\fa}}+i+t-j-1)({d_{\fa}}+i+t-j-2) \cdots ({d_{\fa}}-j)}{(i+t)!} \\ 
\times \frac{({d_{\fa}}-j-1)({d_{\fa}}-j-2) \cdots ({d_{\fa}}-j-m)}{m!} \\
=  \binom{i+t+m}{m} \sum_{j=0}^{{d_{\fa}}-1} \binom{d-i-t+j-1}{j} \binom{{d_{\fa}}+i+t-j-1}{i+t+m}.
\end{multline*}
Since ${d_{\fa}}+i+t-j-1 \ge 0$, the last expression equals 
$$
\binom{i+t+m}{m} \sum_{j=0}^{{d_{\fa}}-1} (-1)^j \binom{-d+i+t}{j} \binom{{d_{\fa}}+i+t-j-1}{{d_{\fa}}-j-1-m}. 
$$

Using the identity $\binom{n}{k}= (-1)^k \binom{k-n-1}{k}$ for the second binomial coefficient in this sum, we can conclude
$$
\left( M_{d_{\fa}}(d,t) \cdot v_m\right)_i=\binom{i+t+m}{m} \sum_{j=0}^{{d_{\fa}}-1} (-1)^{{d_{\fa}}+1+m} \binom{-d+i+t}{j} \binom{-m-i-t-1}{{d_{\fa}}-j-1-m}.
$$
Since ${d_{\fa}}-1-m \le {d_{\fa}}-1$, this is by the Chu-Vandermonde summation \cite[pp.59 -- 60]{Askey} equal to 
$$
(-1)^{{d_{\fa}}+1+m} \binom{i+t+m}{m} \binom{-d-m-1}{{d_{\fa}}-m-1}=\left(w_m(d,t)\right)_i.  
$$
\end{proof}

For the proof of Theorem \ref{thm:eigenvalues}, some more notations are required. 
We let $V_{d_{\fa}}$ be the $d_{\fa}\times d_{\fa}$ matrix with columns $v_j$, $0\leq j\leq d_{\fa}-1$, i.\,e.,  
$$V_{d_{\fa}}:=(v_0,v_1,\ldots,v_{d_{\fa}-1})=\left( \binom{d_{\fa}-i-1}{j} \right)_{0 \le i, j \le d_{\fa}-1}$$ 
Similarly, we collect the column vectors $w_j$, $0\leq j\leq d_{\fa}-1$, in the matrix $W_{d_{\fa}}(d,t)$, i.\,e., 
$$W_{d_{\fa}}(d,t):=(w_0(d,t),w_1(d,t),\ldots,w_{d_{\fa}-1}(d,t)):=\left( (-1)^{d_{\fa}+1+j} \binom{i+t+j}{j} \binom{-d-j-1}{d_{\fa}-j-1} \right)_{0 \le i, j \le d_{\fa}-1}.$$ 
With these notations, Lemma~\ref{lemTrafo} can be rephrased as $M_{d_{\fa}}(d,t) \cdot V_{d_{\fa}} = W_{d_{\fa}}(d,t)$.\\ 
Moreover, we will need the following upper-triangular matrices
$$
A_{d_{\fa}}(d,t):= \left( (-1)^{d_{\fa}+1} \binom{-d_{\fa}-t}{j-i} \binom{-d-j-1}{d_{\fa}-j-1} \right)_{0 \le i, j \le d_{\fa}-1}. 
$$
The three matrices are related as follows.

\begin{lem}\label{lem:vw} Let $d_{\fa}$ be a positive integer. Then the following matrix identity holds
$$W_{d_{\fa}}(d,t)= V_{d_{\fa}} \cdot A_{d_{\fa}}(d,t).$$
\end{lem}

\begin{proof} The claimed identity can be deduced from the Chu-Vandermonde summation \cite[pp.59 -- 60]{Askey} and $\binom{n}{k}= (-1)^{k} \binom{k-n-1}{k}$
as follows:
\begin{align*}
\left( V_{d_{\fa}} \cdot A_{d_{\fa}}(d,t) \right)_{i,j} &= \sum_{k=0}^{d_{\fa}-1} \binom{d_{\fa}-i-1}{k} (-1)^{d_{\fa}+1} \binom{-d_{\fa}-t}{j-k} \binom{-d-j-1}{d_{\fa}-j-1} \\
& =(-1)^{d_{\fa}+1} \binom{-i-t-1}{j} \binom{-d-j-1}{d_{\fa}-j-1} \\
&= (-1)^{d_{\fa}+1+j} \binom{i+j+t}{j} \binom{-d-j-1}{d_{\fa}-j-1} =W_{d_{\fa}}(d,t)_{i,j}.
\end{align*}
\end{proof}

It is easy to check that $V_{d_{\fa}}=(v_{i,j})_{0 \le i, j \le d_{\fa}-1}$ is in fact invertible with  
$$V_{d_{\fa}}^{-1}=\left (v_{d_{\fa}-1-i,d_{\fa}-1-j} (-1)^{i+j+d_{\fa}+1}\right)_{0 \le i, j \le d_{\fa}-1}.$$ 

We can finally provide the proof of Theorem \ref{thm:eigenvalues}.

\begin{proof}[Proof of Theorem~\ref{thm:eigenvalues}]
It follows from the relation $M_{d_{\fa}}(d_i,t_i)=W_{d_{\fa}}(d_i,t_i) \cdot V_{d_{\fa}}^{-1}$ (see the discussion after Lemma~\ref{lemTrafo}) that
$$
M_{d_{\fa}}(d_1,t_1) \cdot M_{d_{\fa}}(d_2,t_2) \cdots M_{d_{\fa}}(d_n,t_n) 
= W_{d_{\fa}}(d_1,t_1) V_{d_{\fa}}^{-1} W_{d_{\fa}}(d_2,t_2) V_{d_{\fa}}^{-1} \cdots W_{d_{\fa}}(d_n,t_n) V_{d_{\fa}}^{-1}.
$$
Now we use $V_{d_{\fa}}^{-1} \cdot W_{d_{\fa}}(d_i,t_i) = A_{d_{\fa}}(d_i,t_i)$ (Lemma~\ref{lem:vw}) to see that the product of matrices on the right-hand side of the last equation is equal to 
$$
W_{d_{\fa}}(d_1,t_1) \cdot A_{d_{\fa}}(d_2,t_2) \cdots A_{d_{\fa}}(d_n,t_n)  \cdot V_{d_{\fa}}^{-1}.
$$
Since conjugating with $V_{d_{\fa}}$ leaves the eigenvalues unchanged, it suffices to compute the eigenvalues of 
$$
A_{d_{\fa}}(d_1,t_1) \cdot A_{d_{\fa}}(d_2,t_2) \cdots A_{d_{\fa}}(d_n,t_n).
$$
Once again we used $V_{d_{\fa}}^{-1} \cdot W_{d_{\fa}}(d_1,t_1) = A_{d_{\fa}}(d_1,t_1)$, see Lemma~\ref{lem:vw}.
The matrices $A_{d_{\fa}}(d_i,t_i)$ are however upper triangular and so is their product. Hence, the eigenvalues of this product 
are just the entries on the main diagonal, which can be computed as the coordinatewise product of the main diagonals of the individual matrices. 
It is easy to verify, that for $0\leq j\leq d_{\fa}-1$ the $j$\textsuperscript{th} diagonal entry of $
A_{d_{\fa}}(d_i,t_i)$ equals $\lambda_{d_{\fa}}(d_i,d_{\fa}-1-j)$ which finishes the proof.
\end{proof}

\section{Eigenvectors of $M_{d_{\fa}}(d,t)$} \label{sect:Eigenspaces}
In this section, we will derive an explicit formula for the eigenvectors of the matrices $M_{d_{\fa}}(d,t)$. For this computation it is relevant 
that the eigenvalues $\lambda_{d_{\fa}}(d,i)=(-1)^{{d_{\fa}}+1+i} \binom{d+{d_{\fa}}-1}{i}$, $i=0,1,\ldots, {d_{\fa}}-1$, of $M_{d_{\fa}}(d,t)$ 
are pairwise distinct, which is true since $d$ is treated as an indeterminate. (Note that if we substitute $d$ by an arbitrary integer, 
certain eigenvalues might coincide. More precisely, we have $\lambda_{d_{\fa}}(d,i)=\lambda_{d_{\fa}}(d,{d_{\fa}}-i+d-1)$ if  $0 \leq i \le \lfloor \frac{d_{\fa} + d - 1}{2} \rfloor$ 
and $d \equiv d_{\fa} + 1 \mod 2$.)

Clearly, it suffices to 
compute the eigenvector of 
$V^{-1}_{d_{\fa}} \cdot M_{d_{\fa}}(d,t)  \cdot V_{d_{\fa}} =
 A_{d_{\fa}}(d,t)$ (the identity follows from the discussion after Lemma~\ref{lemTrafo} and Lemma~\ref{lem:vw}), since $a$ is an eigenvector of 
$A_{d_{\fa}}(d,t)$ with respect to the  eigenvalue $\lambda_{d_{\fa}}(d,i)$ if and only of 
$V_{d_{\fa}} \cdot a$ is an eigenvector of $M_{d_{\fa}}(d,t)$ with respect to the eigenvalue $\lambda_{d_{\fa}}(d,i)$. 

Now, fix $i \in \{0,1,\ldots,d_{\fa}-1\}$ and let $a:=(a_0,a_1,\ldots,a_{{d_{\fa}}-1})^T\in \mathbb{R}^{d_{\fa}}$ be an eigenvector of 
$A_{d_{\fa}}(d,t)$ with eigenvalue $\lambda_{d_{\fa}}(d,i)$. We multiply both sides of $A_{d_{\fa}}(d,t)\cdot a = \lambda_{d_{\fa}}(d,i)\cdot a$ with 
the matrix $X_{d_{\fa}}(t): = \left( \binom{d_{\fa}+t}{k-j} \right)_{0 \le j,k \le d_{\fa}-1}$ from the left. Using the Chu-Vandermonde 
summation, we see that  
$$
\left( X_{d_{\fa}}(t) \cdot A_{d_{\fa}}(d,t) \right)_{j,m} = \sum_{k=0}^{d_{\fa}-1} \binom{d_{\fa}+t}{k-j} \binom{-d_{\fa}-t}{m-k} 
\binom{-d-m-1}{d_{\fa}-m-1} (-1)^{d_{\fa}+1} = \delta_{j,m} \binom{-d-m-1}{d_{\fa}-m-1} (-1)^{d_{\fa}+1}.
$$ This implies that $a$ fulfills
$$
(-1)^{d_{\fa}+1 }\diag_{0 \le j \le d_{\fa}-1} \left( \binom{-d-j-1}{d_{\fa}-j-1}  \right) \cdot a = \lambda_{d_{\fa}}(d,i) X_{d_{\fa}}(t) \cdot a.
$$
Cancelling powers of $-1$, this yields 
\begin{equation} \label{eq2}
a_j\binom{-d-j-1}{{d_{\fa}}-j-1}=(-1)^i\binom{d+{d_{\fa}}-1}{i}\sum_{m=j}^{{d_{\fa}}-1}a_m\binom{{d_{\fa}}+t}{m-j}
\end{equation}
for $j=0,\ldots,{d_{\fa}}-1$, i.\,e.,
\begin{equation}\label{eq:recursion}
a_j=(-1)^{i+d_{\fa}+j+1} \frac{\binom{d+{d_{\fa}}-1}{i}}{\binom{d_{\fa}+d-1}{{d_{\fa}}-j-1}}\sum_{m=j}^{{d_{\fa}}-1}a_m\binom{{d_{\fa}}+t}{m-j},
\end{equation}
where we have used the identity $\binom{n}{k}=(-1)^k\binom{k-n-1}{k}$. 
Solving this equation first for $j={d_{\fa}}-1$ and then subsequently for $j={d_{\fa}}-2,{d_{\fa}}-3, \ldots,1,0$ by backward substitution, it is possible to determine the eigenvector $a$. More precisely, from \eqref{eq:recursion} we get 
\begin{equation*}
a_j\left(1-(-1)^{i+{d_{\fa}}+j+1}\frac{\binom{d+{d_{\fa}}-1}{i}}{\binom{d+{d_{\fa}}-1}{{d_{\fa}}-j-1}}\right)=(-1)^{i+{d_{\fa}}+j+1}\frac{\binom{d+{d_{\fa}}-1}{i}}{\binom{d+{d_{\fa}}-1}{{d_{\fa}}-j-1}}\sum_{m=j+1}^{{d_{\fa}}-1}a_m\binom{{d_{\fa}}+t}{m-j}
\end{equation*}
for $0\leq j\leq {d_{\fa}}-1$. Since, for $j\neq {d_{\fa}}-1-i$ the binomial coefficients $\binom{d+{d_{\fa}}-1}{i}$ and $\binom{d+{d_{\fa}}-1}{{d_{\fa}}-j-1}$ are polynomials of distinct degree in $d$, we conclude --~by backward substitution~-- that $a_j=0$ for ${d_{\fa}}-1\geq j\geq {d_{\fa}}-i$. 
Similarly, for $j={d_{\fa}}-1-i$, it follows that $a_j$ can be chosen arbitrarily in $\mathbb{R}$. 
If ${d_{\fa}}-2-i\geq j\geq 0$, we have 
\begin{equation}\label{eq:final}
a_j=\frac{(-1)^{i+{d_{\fa}}+j+1}\frac{\binom{d+{d_{\fa}}-1}{i}}{\binom{d+{d_{\fa}}-1}{{d_{\fa}}-j-1}}}{1-(-1)^{i+{d_{\fa}}+j+1}\frac{\binom{d+{d_{\fa}}-1}{i}}{\binom{d+{d_{\fa}}-1}{{d_{\fa}}-j-1}}}\sum_{m=j+1}^{{d_{\fa}}-i-1}a_m\binom{{d_{\fa}}+t}{m-j}.
\end{equation}
In order to simplify notation, we set $s_k:=a_{{d_{\fa}}-i-1-k}$ for $0\leq k\leq {d_{\fa}}-i-1$ and
\begin{equation*}
g_k:=\frac{(-1)^{k}\frac{\binom{d+{d_{\fa}}-1}{i}}{\binom{d+{d_{\fa}}-1}{i+k}}}{1-(-1)^{k}\frac{\binom{d+{d_{\fa}}-1}{i}}{\binom{d+{d_{\fa}}-1}{i+k}}} =\left((-1)^k \frac{\binom{d+d_a-1}{i+k}}{\binom{d+d_a-1}{i}}-1\right)^{-1}
\end{equation*}
for $1\leq k\leq {d_{\fa}}-i-1$. Hence, \eqref{eq:final} simplifies to 
\begin{equation*}
s_j=g_j\cdot \sum_{m=0}^{j-1}s_m\binom{{d_{\fa}}+t}{j-m}
\end{equation*}
for $1\leq j\leq {d_{\fa}}-i-1$. 
This can be used to prove the following formula by induction with respect to $j$ ($1\leq j\leq {d_{\fa}}-i-1$):
\begin{equation}
\label{formula:sj}
s_j = s_0  \sum_{m=0}^{j-1}\sum_{0=i_0<i_1<\cdots<i_m<i_{m+1}=j}\prod_{k=1}^{m+1} g_{i_k}\binom{{d_{\fa}}+t}{i_k-i_{k-1}} 
\end{equation}
Indeed,
\begin{align*}
s_j &= g_j \sum_{l=1}^{j-1} s_l \binom{d_{\fa}+t}{j-l}      +  g_j s_0 \binom{d_{\fa}+t}{j}  \\
&= g_j \sum_{l=1}^{j-1} s_0 \sum_{m=0}^{l-1} \sum_{0=i_0 < i_1 < \cdots < i_m < i_{m+1}=l} \left( \prod_{k=1}^{m+1} g_{i_k} 
\binom{d_{\fa}+t}{i_k-i_{k-1}}  \right) \binom{d_{\fa}+t}{j-l} +  g_j s_0 \binom{d_{\fa}+t}{j}
\end{align*}
Now we switch the order of summation of the two outer sums and write $i_{m+1}$ instead of $l$:
\begin{align*}
s_j &=  g_j \sum_{m=0}^{j-2} s_0 \sum_{i_{m+1}=m+1}^{j-1} \sum_{0=i_0 < i_1 < \cdots < i_m < i_{m+1}} \left( \prod_{k=1}^{m+1} g_{i_k} 
\binom{d_{\fa}+t}{i_k-i_{k-1}}  \right) \binom{d_{\fa}+t}{j-i_{m+1}} +  g_j s_0 \binom{d_{\fa}+t}{j} \\
&= s_0 \sum_{m=0}^{j-2}  \sum_{0=i_0 < i_1 < \cdots < i_m < i_{m+1} < i_{m+2}=j} \prod_{k=1}^{m+2} g_{i_k} 
\binom{d_{\fa}+t}{i_k-i_{k-1}}   +  g_j s_0 \binom{d_{\fa}+t}{j} \\
&=  s_0 \sum_{m=1}^{j-1}  \sum_{0=i_0 < i_1 < \cdots  < i_{m} < i_{m+1}=j} \prod_{k=1}^{m+1} g_{i_k} 
\binom{d_{\fa}+t}{i_k-i_{k-1}}   +  g_j s_0 \binom{d_{\fa}+t}{j}
\end{align*}
This proves \eqref{formula:sj}.

Switching from $s_j$ back to $a_j$ by using the relation $a_j=s_{{d_{\fa}}-i-1-j}$ for $0\leq j \leq {d_{\fa}}-i-1$, we can conclude
\begin{equation*}
a_j=s_{{d_{\fa}}-i-1-j}= a_{d_{\fa}-i-1}\cdot\sum_{m=0}^{{d_{\fa}}-i-2-j}\sum_{0=i_0<i_1<\cdots<i_m<i_{m+1}={d_{\fa}}-i-1-j}\prod_{k=1}^{m+1} g_{i_k}\binom{{d_{\fa}}+t}{i_k-i_{k-1}}
\end{equation*} 
for $0\leq j\leq {d_{\fa}}-i-2$, where $a_{d_{\fa}-i-1}\in\mathbb{R}$ is arbitrary. E.\,g., we can set $a_{d_{\fa}-i-1}=1$. 
We summarize the above discussion and part of the results of Section \ref{sect:Spectrum} in the following proposition.
\begin{prop}\label{prop:eigenvectors}
Let ${d_{\fa}}\in\mathbb{N}$ and let $d,t$ be indeterminates. Then:
\begin{itemize}
\item[(i)] The matrix $A_{d_{\fa}}(d,t)$ has eigenvalues $\lambda_{d_{\fa}}(d,i)=(-1)^{{d_{\fa}}+1+i} \binom{d+{d_{\fa}}-1}{i}$ for $0\leq i\leq {d_{\fa}}-1$.
\item[(ii)] For $0\leq i\leq {d_{\fa}}-1$ the eigenspace of $A_{d_{\fa}}(d,t)$ for the eigenvalue $\lambda_{d_{\fa}}(d,i)$ is spanned by the vector $a=(a_0,\ldots,a_{{d_{\fa}}-1})^T$ defined by
\begin{align*}
a_j=
\begin{cases}
0, \qquad &\mbox{for } {d_{\fa}}-i\leq j\leq {d_{\fa}}-1\\
1, \qquad &\mbox{for } j={d_{\fa}}-i-1\\
\sum_{m=0}^{{d_{\fa}}-i-2-j}\sum_{0=i_0<i_1<\cdots<i_m<i_{m+1}={d_{\fa}}-i-1-j}\prod_{k=1}^{m+1}\left(g_{i_k}\binom{{d_{\fa}}+t}{i_k-i_{k-1}}\right), &\mbox{for } 0\leq j\leq {d_{\fa}}-i-2.
\end{cases}
\end{align*}
Here the coefficients $g_1,\ldots,g_{{d_{\fa}}-i-1}$ are given as 
\begin{equation*}
g_k=\left((-1)^k \frac{\binom{d+d_a-1}{i+k}}{\binom{d+d_a-1}{i}}-1\right)^{-1}.
\end{equation*}
\end{itemize}
\end{prop}
From this proposition it is straightforward to derive the eigenvectors of the matrices $M_{d_{\fa}}(d,t)$. Indeed, as already mentioned, 
the eigenvectors of $M_{d_{\fa}}(d,t)$ can be obtained from those of $A_{d_{\fa}}(d,t)$ by multiplying them by $V_{d_{\fa}}$ from the left. 
Thus, if $a=(a_0,\ldots,a_{{d_{\fa}}-1})^T$ is the eigenvector of $A_{d_{\fa}}(d,t)$ corresponding to the eigenvalue $\lambda_{d_{\fa}}(d,i)$ 
as defined in Proposition \ref{prop:eigenvectors}, then $b:=V_{d_{\fa}} \cdot a$ is eigenvector of $M_{d_{\fa}}(d,t)$ corresponding to the 
eigenvalue $\lambda_{d_{\fa}}(d,t)$. Note, that since the vector $a$ is not the zero vector and since $V_{d_{\fa}}$ is invertible, 
we can conclude that the vector $b$ is not the zero vector. 
We finally compute the vector $b=V_{d_{\fa}} \cdot a:=(b_0,\ldots,b_{{d_{\fa}}-1})^T$. For $0\leq k\leq {d_{\fa}}-1$ the $k$\textsuperscript{th} entry of $b$ is given by
\begin{align*}
b_k&=\sum_{j=0}^{{d_{\fa}}-i-1}a_j\binom{{d_{\fa}}-k-1}{j}\\
&=\binom{{d_{\fa}}-k-1}{{d_{\fa}}-i-1}+\sum_{j=0}^{{d_{\fa}}-i-2}\binom{{d_{\fa}}-k-1}{j}\left(\sum_{m=0}^{{d_{\fa}}-i-2-j}\sum_{0=i_0<i_1<\cdots<i_m<i_{m+1}={d_{\fa}}-i-1-j}\prod_{k=1}^{m+1} g_{i_k}\binom{{d_{\fa}}+t}{i_k-i_{k-1}} \right).
\end{align*}
We conclude this section by summarizing the main results in the following theorem.
\begin{theo} 
Let ${d_{\fa}}\in\mathbb{N}$ and let $d,t$ be indeterminates. 
For $0\leq i\leq {d_{\fa}}-1$ the eigenspace of $M_{d_{\fa}}(d,t)$ for the eigenvalue $\lambda_{d_{\fa}}(d,i)$ is spanned by the vector $b=(b_0,\ldots,b_{{d_{\fa}}-1})$ defined by 
\begin{equation*}
b_k=\binom{{d_{\fa}}-k-1}{{d_{\fa}}-i-1}+\sum_{j=0}^{{d_{\fa}}-i-2}\binom{{d_{\fa}}-k-1}{j}\left(\sum_{m=0}^{{d_{\fa}}-i-2-j}\sum_{0=i_0<i_1<\cdots<i_m<i_{m+1}={d_{\fa}}-i-1-j}\prod_{k=1}^{m+1} g_{i_k}\binom{{d_{\fa}}+t}{i_k-i_{k-1}} \right)
\end{equation*}
for $0\leq k\leq {d_{\fa}}-1$. 
Here the coefficients $g_1,\ldots,g_{{d_{\fa}}-i-1}$ are given as
\begin{equation*}
g_k=\left((-1)^k \frac{\binom{d+d_a-1}{i+k}}{\binom{d+d_a-1}{i}}-1\right)^{-1}.
\end{equation*}
\end{theo}

\section{Discussion}\label{sect:conclusion}
The aim of this section is to collect some easy properties of the $h$-vector transformation of the Segre product. 
We will conclude this section with an open question.

It is easy to see that the following properties are preserved under taking the Segre product of sequences.

\begin{prop}
Let $\fa=(a_n)_{n\geq 0}$ and $\fb=(b_n)_{n\geq 0}$ sequences of complex numbers such that their generating series are of the form \eqref{eq:powerSeries}.
\begin{itemize}
\item[(i)] If $h(\fa)$ and $h(\fb)$ are non-negative (entrywise), then so is $h(\fa\ast\fb)$.
\item[(ii)] Let $m_{\fa}$ and $m_{\fb}$ be the last non-zero entry of the $h$-vector of $\fa$ and $\fb$, respectively, and assume that $d_{\fa} - m_{\fa} = d_{\fb} - m_{\fb}$. If $h(\fa)$ and $h(\fb)$ are symmetric, i.\,e., $h_i(\fa)=h_{m_{\fa}-i}(\fa)$ and $h_i(\fb)=h_{m_{\fb}-i}(\fb)$ for all $i\geq 0$, then so is $h(\fa\ast\fb)$.
\end{itemize}
\end{prop}

\begin{proof}
Part (i) holds since all entries of the transformation matrix are non-negative.

Part (ii) follows from Corollary \ref{cor:sym}, the symmetrie of $h(\fa)$ and $h(\fb)$ and from \eqref{eq:trafo}.
\end{proof}

In \cite{BW-barycentric} and \cite{BrentiWelker} the $h$-vector transformations when passing from a simplicial complex to its barycentric subdivision and from a standard graded algebra to its $r$\textsuperscript{th} Veronese algebra are studied. In these article, the transformations are not only described explicitly but also the asymptotic behavior of the roots of the $h$-polynomials is worked out. In particular, it is shown that, in both cases, the $h$-polynomial becomes real-rooted after applying the barycentric subdivision operation once or a high enough Veronese algebra. It is natural to ask if a similar statement might be true for (high enough) Segre products of standard graded algebras respectively number sequences. Motivated by results of certain computer experiments, we propose the following conjecture.

\begin{con}\label{con:realRoots}
Let $\fa=(a_n)_{n\geq 0}$ be a sequence of complex numbers such that its generating series is of the form 
\begin{equation*}
\fa(t)=\sum_{n\geq 0}a_n t^n =\frac{h_0(\fa)+\cdots + h_{d_{\fa}-1}(\fa)t^{d_{\fa}-1}}{(1-t)^{d_{\fa}}},
\end{equation*}
where $h_i(\fa)\geq 0$ for all $0\leq i\leq d_{\fa}-1$. Let $\fa^{\ast r}(t)=\underbrace{(\fa\ast\cdots\ast \fa)}_{r \mbox{\tiny{ times}}}(t)$ be the $r$-th Segre product of the sequence $\fa$. Then, there exists a positive integer $R$ such that for $r\geq R$, the polynomial 
\begin{equation*}
\fh(\fa^{\ast r})(t)=\sum_{i\geq 0}^{r(d_{\fa}-1)}h_i(\fa^{\ast r}) t^i
\end{equation*}
has only non-positive real-roots.
\end{con}

We want to remark that the techniques from \cite{BW-barycentric} and \cite{BrentiWelker} are not applicable in the above setting since those 
treat sequences of polynomials $(p_r)_{r\geq 0}$ such that all $p_r$ are of the same degree. However, in our case the degree of the polynomials 
$\fh(\fa^{\ast r})(t)$ grows linearly in $r$, which makes the situation more complicated. Also other results on real-rootedness of polynomials,
 see e.\,g., \cite{BorceaBraenden,Brenti,Wagner} and references therein, usually require the considered polynomials to be of equal degree so
that they cannot be used for proving Conjecture \ref{con:realRoots}.

\bibliography{biblio}
  \bibliographystyle{plain}

\end{document}